\documentclass[11pt]{amsart}
\usepackage{amsmath,amssymb,amsthm,latexsym,cite,cancel}
\usepackage[small]{caption}
\usepackage{graphicx,wasysym,overpic,tikz,color}
\usepackage{subfigure}
\usepackage{cite}
\usepackage[colorlinks=true,urlcolor=blue,
citecolor=red,linkcolor=blue,linktocpage,pdfpagelabels,
bookmarksnumbered,bookmarksopen]{hyperref}
\usepackage[italian,english]{babel}
\usepackage{units}
\usepackage{enumitem}
\usepackage[left=2.1cm,right=2.1cm,top=2.71cm,bottom=2.71cm]{geometry}
\usepackage[hyperpageref]{backref}
\usepackage{float,ulem}

\usepackage[colorinlistoftodos]{todonotes}
\makeatletter
\def\namedlabel#1#2{\begingroup
    #2%
    \def\@currentlabel{#2}%
    \phantomsection\label{#1}\endgroup
}
\makeatother

\newtheorem{theorem}{Theorem}[section]
\newtheorem{definition}[theorem]{Definition}
\newtheorem{proposition}[theorem]{Proposition}
\newtheorem{lemma}[theorem]{Lemma}

\newtheorem{corollary}[theorem]{Corollary}

\newcommand{\abs}[1]{\lvert#1\rvert}
\newcommand{\norm}[1]{\lVert#1\rVert}

\newcommand{\R}{\mathbb{R}}

\tikzstyle{nodo}=[circle,draw,fill,inner sep=0pt,minimum size=%
1.5mm]

\numberwithin{equation}{section}

\title[Normalized solutions to nonlinear Schr\"odinger equations]{Existence of normalized solutions to nonlinear Schr\"odinger equations on lattice graphs}

\author[Z. He]{Zhentao He}

\address[Z. He]{\newline\indent
	School of Mathematics
	\newline\indent
	East China University of Science and Technology
	\newline\indent
	Shanghai 200237, PR China }
\email{\href{mailto:hezhentao2001@outlook.com}{hezhentao2001@outlook.com}}

\author[C. Ji]{Chao Ji}

\address[C. Ji]{\newline\indent
	School of Mathematics
	\newline\indent
	East China University of Science and Technology
	\newline\indent
	Shanghai 200237, PR China }
\email{\href{mailto:jichao@ecust.edu.cn}{jichao@ecust.edu.cn}}

\author[Y. Tao]{Yifan Tao}

\address[Y. Tao]{\newline\indent
	School of Mathematics
	\newline\indent
	East China University of Science and Technology
	\newline\indent
	Shanghai 200237, PR China }
\email{\href{mailto:taoyf2002@126.com}{taoyf2002@126.com}}

\date{\today}

\begin{document}
\begin{abstract}
\iffalse In this paper, using a discrete Schwarz rearrangement on lattice graphs developed in \cite{DSR}, we study the existence of positive normalized solutions to the nonlinear Schr\"odinger equations on lattice graphs. In particular, for $N \geq 2$, given mass $m>0$,  $f \in C(\mathbb{R}, \mathbb{R})$ satisfying some technical assumptions and $F(t):=\int_0^t f(\tau) \,d\tau$, and for the corresponding energy functional $I:H^1\left(\mathbb{Z}^N\right)\to \R$,
 $$I(u)=\frac{1}{2} \int_{\mathbb{Z}^N}|\nabla u|^2 \,d\mu-\int_{\mathbb{Z}^N} F(u)\, d\mu,$$
we prove the following minimization problem
$$
\inf_{u \in S_m} I(u)
$$
 has an excitation threshold $m^*\in [0,+\infty]$ such that
\begin{equation*}
    \inf_{u \in S_m} I(u)<0 \quad  \text{if and only if } m>m^*.
\end{equation*}
Based primarily on $m^* \in (0,+\infty)$ or $m^*=0$, we classify the problem into $3$ different cases: $L^2$-subcritical, $L^2$-critical and $L^2$-supercritical. Moreover, for all $3$ cases, under assumptions that we believe to be nearly optimal,  we show that, when $m>m^*$, the  minimization problem is achieved by some positive Schwarz symmetric function. \fi

In this paper, using a discrete Schwarz rearrangement on lattice graphs developed in \cite{DSR}, we study the existence of  global minimizers for the following functional $I:H^1\left(\mathbb{Z}^N\right)\to \R$,
$$I(u)=\frac{1}{2} \int_{\mathbb{Z}^N}|\nabla u|^2 \,d\mu-\int_{\mathbb{Z}^N} F(u)\, d\mu,$$
constrained on $S_m:=\left\{u \in H^1\left(\mathbb{Z}^N\right) \mid\|u\|_{\ell^2\left(\mathbb{Z}^N\right)}^2=m\right\}$,
where $N \geq 2$, $m>0$ is prescribed, $f \in C(\mathbb{R}, \mathbb{R})$ satisfying some technical assumptions and $F(t):=\int_0^t f(\tau) \,d\tau$. We prove the following minimization problem
$$
\inf_{u \in S_m} I(u)
$$
 has an excitation threshold $m^*\in [0,+\infty]$ such that
\begin{equation*}
    \inf_{u \in S_m} I(u)<0 \quad  \text{if and only if } m>m^*.
\end{equation*}
Based primarily on $m^* \in (0,+\infty)$ or $m^*=0$, we classify the problem into three    different cases: $L^2$-subcritical, $L^2$-critical and $L^2$-supercritical. Moreover,  for all three cases, under assumptions that we believe to be nearly optimal, we show that $m^*$ also separates the existence and nonexistence of global minimizers for $I(u)$ constrained on $S_{m}$.
\end{abstract}
\maketitle
\section{Introduction}
In this paper, we consider the existence of normalized solutions to the nonlinear Schr\"odinger equation on lattice graphs
\begin{equation}\label{eqpm}
\left\{\begin{aligned}
&-\Delta u  =f(u)+\lambda u \quad \text { in } \mathbb{Z}^N, \\
&\|u\|_{\ell^2\left(\mathbb{Z}^N\right)}^2  =m, \\
&u  \in H^1\left(\mathbb{Z}^N\right),
\end{aligned}\right.\tag{$P_{m}$}
\end{equation}
where $N \geq 2$, $f \in C(\mathbb{R}, \mathbb{R})$, $m>0$ is prescribed,  and the parameter $\lambda \in \mathbb{R}$ is part of the unknowns which arises  as a Lagrange multiplier.

Under mild assumptions on $f$, one can define a $C^1$ functional $I: H^1\left(\mathbb{Z}^N\right) \rightarrow \mathbb{R}$ by
$$
I(u):=\frac{1}{2} \int_{\mathbb{Z}^N}|\nabla u|^2 \,d\mu-\int_{\mathbb{Z}^N} F(u)\, d\mu,
$$
where $F(t):=\int_0^t f(\tau) \,d\tau$ for $t \in \mathbb{R}$. It is clear that solutions to problem \eqref{eqpm} correspond to critical points of
$I$ constrained to
$$
S_m:=\left\{u \in H^1\left(\mathbb{Z}^N\right) \mid\|u\|_{\ell^2\left(\mathbb{Z}^N\right)}^2=m\right\}.
$$
For future reference, the value $I(u)$ is called the energy of $u$.

In the past decade,  a vast literature is dedicated to the existence, multiplicity and stability of normalized solutions for the following problem
\begin{align}\label{P1}
	\left\{
	\begin{aligned}
		&-\Delta u= \lambda u + f(u) \textrm{ \ in \ } \R^N,\\
		& u \in H^{1}(\R^N), \\
		&\int_{\R^N}u^{2}dx=m,
	\end{aligned}
	\right.
\end{align}
where $N \geq 1$, $f \in C(\mathbb{R}, \mathbb{R})$, $m>0$ is prescribed,  and the parameter $\lambda \in \mathbb{R}$ is part of the unknowns which arises  as a Lagrange multiplier.
Problem \eqref{P1} is strongly related with the following time-dependent, nonlinear Schr\"{o}dinger equation
\[
\left\{
\begin{array}{ll}
i \dfrac{\partial \Psi}{\partial t}(t,x) = \Delta_x \Psi(t,x) + h(|\Psi(t,x)|)\Psi(t,x), & (t,x) \in \mathbb{R} \times \Omega, \\
\int_{\Omega} |\Psi(t,x)|^2\,dx = m &
\end{array}
\right.
\]
where $f(u)=h(|u|)u$, the prescribed mass \(\sqrt{m}\) appears in nonlinear optics and the theory of Bose-Einstein condensates (see \cite{Fibich, TVZ, Zhang}). Solutions \(u\) to \eqref{P1} correspond to standing waves \(\Psi(t,x) = e^{-i\lambda t} u(x)\) of the foregoing time-dependent equation. The prescribed mass represents the power supply in nonlinear optics or the total number of particles in Bose-Einstein condensates.

 In the seminar paper \cite{CL}, for the case of $L^{2}$-subcritical, Cazenave  and Lions proved the existence of normalized ground states by the minimization method and presented a general method which enables us to prove the orbital stability of some standing waves  for problem \eqref{P1}. However, the $L^{2}$-supercritical problem is totally different. One key difference between the $L^2$-subcritical case and $L^2$-supercritical case is that, for $L^2$-subcritical case, the energy functional $\bar{I}(u):=\int_{\R^N}\abs{\nabla u}^2\, dx-\int_{\R^N}F(u)\, dx$ is bounded from below on $\bar{S}_m:=\{u \in H^1(\R^N): \int_{\R^N}\abs{u}^2\, dx=m\}$, but, for $L^2$-supercritical case, one can show that  $\inf_{\bar{S}_m}\bar{I}(u)=-\infty$ by a scaling technique. In \cite{jeanjean1}, Jeanjean introduced the fibering map
$$
\tilde{I}(u,s) = \frac{e^{2s}}{2}\int_{\R^N}\abs{\nabla u}^2\, dx-\frac{1}{e^{sN}}\int_{\R^N}F(e^\frac{sN}{2}u(x))\, dx,
$$
where $F(t):=\int_0^tf(s)\,ds$. Using the scaling technique and Poho\v{z}aev identity, the author proved that $\bar{I}$ on $\bar{S}_m$ and $\tilde{I}$ on $\bar{S}_m \times\R$ satisfy the mountain pass geometry and their mountain pass levels are equal, then the author construct a bounded Palais-Smale sequence on $\bar{S}_m$ and obtained a radial solution to problem \eqref{P1}. It is worth noting that the scaling technique plays a crucial role  in the study of normalized solutions to problem \eqref{P1}. For further results related to \eqref{P1},  see \cite{ WeiWu, CCM,  AlvesThin, BartschJeanjeanSaove, Bartschmolle, BartschSaove, Bartosz, BellazziniJeanjeanLuo, CazenaveLivro, Jun, jeanjean1, JeanjeanJendrejLeVisciglia, JeanjeanLu, JeanjeanLu2020, JeanjenaLu2, JeanjeanLe, Nicola1, Nicola2} and the references therein.

In recent years, the various mathematical problems on graphs have been extensively investigated (see \cite{Gr, Gri1, hua2021existence, HX1, HLY, DSR, Stefanov,Weinstein,He}  and references therein).  In particular, in the monograph \cite{Gr}, Grigor'yan introduced the discrete Laplace operator on finite and infinite graphs; in \cite{HLY}, Huang, Lin and Yau proved two existence results of solutions to mean field equations on an arbitrary connected finite graph; in \cite{HX1}, applying the Nehari method, Hua and Xu studied the existence of ground states of the nonlinear Schr\"{o}dinger equation $-\Delta u+V(x) u=f(x, u)$ on the lattice graph $\mathbb{Z}^N$ where $f$ satisfies some growth conditions and the potential function $V$ is periodic or bounded. Due to our scope, we would like to mention the seminar paper \cite{Weinstein}, which appears to be the first to  consider problem $\inf_{u \in S_{m}} I(u)$. In \cite{Weinstein}, Weinstein studied the existence of solutions to problem $\inf_{u \in S_{m}} I(u)$ in the case  $f(u) =\abs{u}^{p-2}u$ with $p>2$. He proved the existence of an excitation threshold $m^*\in [0,+\infty)$ such that
\begin{equation}\label{eqet}
    \inf_{u \in S_{m}} I(u)<0 \quad  \text{if and only if } m>m^*.
\end{equation}
For $2<p<2+\frac{4}{N}$, it was shown that $m^*=0$, and then used minimization method to rule out the  vanishing and dichotomy cases, thus obtaining a solution to the problem $\inf _{u \in S_{m}} I(u)$; for $p\geq 2+\frac{4}{N}$,  the author proved that $m^*\in (0,+\infty)$, and again employed the minimization method to show that the problem $\inf _{u \in S_{m}} I(u)$ admits a solution if $m>m^*$,  but has no solution if $0<m<m^*$. However, it is worth pointing out that the proofs of problem $\inf_{u \in S_{m}} I(u)$ given in \cite[Theorem 7.1]{Weinstein} for the case $p\geq 2+\frac{4}{N}$ are incomplete.  In particular, the author excluded the dichotomy cases by a strict subadditivity inequality (see \cite{Lions}) in \cite[Theorem 7.1]{Weinstein}, but, in \cite[Theorem I.1]{Lions}, the strict subadditivity inequality was established under the assumption that $\inf _{u \in S_{\nu}} I(u)<0$ for all $0<\nu \leq m$, and thus, the inequality cannot hold for $p\geq 2+\frac{4}{N}$ on lattice graphs, because by \eqref{eqet}, we have $\inf _{u \in S_{\nu}} I(u)\geq 0$ for all $0 <\nu\leq m^*$.

For problem \eqref{P1} with $f(u) =\abs{u}^{p-2}u$ and  $2<p<2^*$, where $2^*:= \frac{2N}{N-2}$ if $N \geq 3$ and $2^* := +\infty$ if $N = 1, 2$, the classification of the exponent $p$ is based on whether the energy functional $\bar{I}$ is bounded from below on $\bar{S}_m$:
\begin{enumerate}[label=(\roman*)]
\item When $2<p<2+\frac{4}{N}$, the functional $\bar{I}$ is bounded from below on $\bar{S}_m$ and we say that the problem is purely \textbf{$L^2$-subcritical};
\item When $p=2+\frac{4}{N}$, there exists $m_1>0$ such that, $$\inf _{u \in \bar{S}_{m}} \bar{I}(u)=\begin{cases}
    0 \quad &\text{for } m \in (0,m_1]\\
    -\infty \quad&\text{for }m\in (m_1,+\infty),
\end{cases}$$
as shown in detail in \cite[Section 2.2]{CGIT}, and we say that the problem is purely \textbf{$L^2$-critical};
\item When $2+\frac{4}{N}<p<2^*$, the functional  $\bar{I}$ is not bounded from below on $\bar{S}_m$ and we say that the problem is purely \textbf{$L^2$-supercritical}.
\end{enumerate}
However, the problem \eqref{eqpm} on lattice graphs is fundamentally different from problem \eqref{P1}. In fact, as shown in Lemma \ref{lembounded} below, the functiona $I$ is bounded from below on $S_m$ for all $m>0$. Thus, classifying the values of $p$ based on whether $I$ is bounded from below on $S_m$ is no longer applicable.  For the convenience of future research, inspired by \cite{Weinstein}, we classify the exponent
$p$ for problem \eqref{eqpm} on lattice graphs using the same terminology as in the Euclidean setting,  based primarily on $m^* \in (0,+\infty)$ or $m^*=0$. Specifically, for $f(u) =\abs{u}^{p-2}u$ with  $p>2$:
\begin{enumerate}[label=(\roman*)]
\item When $2<p<2+\frac{4}{N}$, $m^*=0$ and the problem  is said to be purely \textbf{$L^2$-subcritical};
\item When $p=2+\frac{4}{N}$, $m^* \in (0,+\infty)$ and  the problem is termed  purely \textbf{$L^2$-critical};
\item When $p>2+\frac{4}{N}$, $m^* \in (0,+\infty)$ and we say that the problem is purely \textbf{$L^2$-supercritical}.
\end{enumerate}
More generally, in a view of Theorem \ref{th1} (iv) and Theorem \ref{th2}, we may classify the nonlinearity $f$ as follows:
\begin{enumerate}[label=(\roman*)]
\item If $f$ satisfies $\displaystyle \lim_{t \to 0}\frac{F(t)}{\abs{t}^{2+\frac{4}{N}}} = +\infty$, then $m^*=0$ and the problem is said to be \textbf{$L^2$-subcritical};
\item If $f$ satisfies $\displaystyle \limsup_{t \to 0}\frac{F(t)}{\abs{t}^{2+\frac{4}{N}}} < +\infty$ but $\frac{F(t)}{\abs{t}^{2+\frac{4}{N}}}$ fails to converge to $0$ as $t \to 0$, then $m^* \in (0,+\infty)$ and  the problem is said to be \textbf{$L^2$-critical};
\item If $\displaystyle \lim_{t \to 0}\frac{F(t)}{\abs{t}^{2+\frac{4}{N}}} = 0$, then $m^* \in (0,+\infty)$ and  the problem is said to be \textbf{$L^2$-supercritical}.
\end{enumerate}

In \cite{DSR}, Hajaiej, Han and Hua developed a discrete Schwarz rearrangement on lattice graphs. In particular, in \cite[Section 6.1]{DSR}, using the discrete Schwarz rearrangement, for all $m>0$, the authors proved the existence of solution to minimization problem
$$
\inf _{u \in S_{m}} \hat{I}(u),
$$
where
$$
\hat{I}(u)=\frac{1}{2} \int_{\mathbb{Z}^N}|\nabla u|^2 \,d\mu-\int_{\mathbb{Z}^N} F(x,u)\, d\mu
$$
with $F(x,t):=\int_0^tf(x,\tau)\,d\tau$ and $f(x,t)$ satisfying $L^2$-subcritical growth  according to our classification.
%, but the case where $f(x,u)$ staisfies $L^2$-critical and $L^2$-supercritical growth is not covered by the authors' assumptions.
It is natural to ask whether there exist solutions of  equation \eqref{eqpm} in the $L^2$-critical and $L^2$-supercritical cases with general $f \in C(\R,\R)$. However, for the $L^2$-supercritical case, applying the method developed in \cite{jeanjean1}-originally used to obtain solutions to equation \eqref{P1} on $\mathbb{R}^N$- to the lattice graphs, some new difficulties arise. In particular, the scaling technique---commonly used in the analysis on $\mathbb{R}^N$---is no longer applicable, and the Poho\v{z}aev identity on lattice graphs remains unknown.

Motivated by the aforementioned works, in this paper we investigate the existence of solutions to the minimization problem $\inf _{u \in S_{m}} I(u)$  including $L^2$-critical and $L^2$-supercritical cases by using  a discrete Schwarz rearrangement argument developed in \cite[Section 6.1]{DSR}.
 Indeed, in Lemma \ref{lemachieved}, when $\inf _{u \in S_{m}} I(u)<0$, we prove the existence of positive solutions of the minimization problem $\inf _{u \in S_{m}} I(u)$ without using the strict subadditivity inequality.

We denote by $\mathbb{Z}^{N}$ the standard lattice graphs with the set of vertices
$$\left\{x=(x_{1},...,x_{N}):x_{i}\in\mathbb{Z},1\le i\le N\right\},$$
and the set of edges
\[E=\left\{(x,y):x,y\in\mathbb{Z}^{N},\sum_{i=1}^{N}|x_{i}-y_{i}|=1\right\}.\]
We denote the space of functions on $\mathbb{Z}^{N}$ by $C(\mathbb{Z}^{N})$. For $u\in C(\mathbb{Z}^{N})$, its support set is defined as $\operatorname{supp}(u):=\{x\in\mathbb{Z}^{N}:u(x)\neq 0\}$. Let $C_{c}(\mathbb{Z}^{N})$ be the set of all functions with finite support. Let $\mu$ be the counting measure on $\mathbb{Z}^{N}$, i.e., for any subset $A \subset \mathbb{Z}^{N}$, $\mu(A):=\#\{x: x \in A\}$. For any function $f$ on $\mathbb{Z}^{N}$, we write
$$
\int_{\mathbb{Z}^{N}} f d \mu:=\sum_{x \in \mathbb{Z}^{N}} f(x)d\mu,
$$
whenever it makes sense.

For any $1 \le p \le  \infty$, $\ell^p(\mathbb{Z}^{N})$ denotes the linear space of $p$-th integrable functions on $\mathbb{Z}^{N}$ equipped with the norm
$$
\|u\|_{\ell^p(\mathbb{Z}^{N})}:=
\begin{cases}
	\left( \sum_{x \in \mathbb{Z}^{N}} |u(x)|^p \right)^{1/p}, & 1 \le p <\infty, \\
	\sup_{x \in \mathbb{Z}^{N}} |u(x)|, & p = \infty.
\end{cases}
$$
In this paper, we shall write $\|u\|_{\ell^{p}(\mathbb{Z}^{N})}$ as $\|u\|_{p}$ for convenience, when there is no confusion.

For any function $u, v \in C(\mathbb{Z}^{N})$, we define the associated gradient form as
$$
\Gamma(u, v)(x):= \frac{1}{2} \sum_{y \sim x} (u(y)-u(x))(v(y)-v(x)) .
$$
Let $\Gamma(u):=\Gamma(u, u)$, and define
\begin{equation}\label{tidu}
	|\nabla u|(x):=\sqrt{\Gamma(u)(x)}=\left( \frac{1}{2} \sum_{y \sim x}(u(y)-u(x))^2\right)^{1 / 2} .
\end{equation}
Let $H^1(\mathbb{Z}^{N})$ be the completion of $C_c(\mathbb{Z}^{N})$ under the norm
$$\|u\|_{H^{1}(\mathbb{Z}^{N})}:=\left(\frac{1}{2} \sum_{x \in \mathbb{Z}^{N}} \sum_{y \sim x}(u(y)-u(x))^2+\sum_{x \in \mathbb{Z}^{N}}  u^2(x)\right)^{1 / 2}.$$
In this paper, we shall write $\|u\|_{H^{1}(\mathbb{Z}^{N})}$ as $\|u\|$ for convenience, when there is no confusion.
Our goal in this paper is to consider the existence of positive solutions to the minimization problem
$$
E_m=\inf _{u \in S_{m}} I(u).
$$
To this end, we will impose the following assumptions on the nonlinearity  $f \in C(\mathbb{R}, \mathbb{R})$:
\begin{itemize}[label=$f\arabic*$]
\item[(\namedlabel{f1}{f$_1$})]  $\displaystyle \lim_{t \rightarrow 0} \frac{f(t)}{t}=0$;
\item[(\namedlabel{f2}{f$_2$})]$\frac{F(t)}{t^2}$ is non-decreasing on $(-\infty,0)\cup(0,+\infty)$;
\item[(\namedlabel{f3}{f$_3$})] $F(t)\leq F(\abs{t})$ for all $t\in \R$;
  \item[(\namedlabel{f4}{f$_4$})] there exists $\zeta>0$ such that $F(\zeta)-2N\zeta^2>0$;
  \item[(\namedlabel{f5}{f$_5$})] $\displaystyle\limsup_{t \to 0}\frac{F(t)}{\abs{t}^{2+\frac{4}{N}}}<+\infty$;
   \item[(\namedlabel{f6}{f$_6$})]$\frac{F(t)}{t^2}$ is strictly increasing on $(-\infty,0)\cup(0,+\infty)$;
  \item[(\namedlabel{f7}{f$_7$})] $\displaystyle\lim_{t \to 0}\frac{F(t)}{\abs{t}^{2+\frac{4}{N}}}=+\infty$.
\end{itemize}
It is clear that \eqref{f2} is a weaker version of \eqref{f6}, and \eqref{f4} is a weaker version of $$
\lim_{t \rightarrow +\infty}\frac{F(t)}{t^2}=+\infty.
$$
We would like to highlight here that there are a great many functions $f$ satisfying (\ref{f1})-(\ref{f4}), for example,
$$f(t)=\abs{t}^{s-2}t\quad  \text{ with } s>2,$$
$$f(t)=\abs{t}^{s_1-2}t + \abs{t}^{s_2-2}t\quad  \text{ with } s_2,s_1>2,$$
$$f(t)=\abs{t}^{s_1-2}t + \abs{t}^{s_2-2}t\ln(1+\abs{t})\quad  \text{ with } s_2,s_1>2,$$
and
$$
f(t)=(e^{\abs{t}}-\abs{t}-1)\operatorname{sgn}t.
$$
\begin{theorem}\label{th1}
Assume that $f$ satisfies \eqref{f1}. $E_m> -\infty$ and the mapping $m \mapsto E_m$ is non-increasing and continuous.
Moreover,
\begin{enumerate}[label=(\roman*)]
\item there exists a uniquely determined number $m^* \in[0, +\infty]$ such that
$$
E_m=0 \text { if } 0<m \leq m^* ( 0<m<+\infty\text{ when } m^*=+\infty), \quad \text{ and }\quad E_m<0 \text { if } m>m^* ;
$$
\item if in addition \eqref{f2} and \eqref{f3} hold, then, when $m>m^*$, the infimum $E_m$ is achieved by some positive Schwarz symmetric (as defined in Definition \ref{defss}) function $u \in S_m$;
\item if in addition \eqref{f4} hold, then $m^* \in [0,\zeta^2)$;
\item if in addition \eqref{f5} hold, then $m^*\in (0,+\infty]$;
\item if in addition \eqref{f6} hold, then, when $0<m<m^*$, $E_m$ is not achieved.
\end{enumerate}
\end{theorem}

We note that, in \cite[Proposition 4.2]{Weinstein}, for $f(t)=\abs{t}^{p-2}t$ with $p\geq 2+\frac{4}{N}$, Weinstein proved
        $$
        m^*= \left(\frac{p}{2}J^{p,N}\right)^\frac{2}{p-2},
        $$
        where
        $$
        J^{p,N}:=\inf_{u \in H^1(\mathbb{Z}^N)\backslash\{0\}}\frac{\norm{u}_2^{{p-2}}\norm{\nabla u}_2^2}{\norm{u}_p^p}.
        $$
Moreover, in a view of Theorem \ref{th1}, we have
\begin{corollary}
Assume that $f$ satisfies \eqref{f1}-\eqref{f4}. Then, there exists $m^* \in [0,\zeta^2)$ such that, for all $m>m^*$, the infimum $E_m$ is achieved by some positive Schwarz symmetric (as defined in Definition \ref{defss}) function $u \in S_m$.
\end{corollary}
The next theorem considering $L^2$-subcritical case and covers the results in \cite[Theorem 6.1]{DSR} when $f(x,u)=f(u)$. The argument involving scaling technique due to \cite{JeanjeanLu,JeanjenaLu2} is not available, however, we made the surprising discovery that, in the $L^2$-subcritical case, the assumption \eqref{f2} (or (F$_5$) in \cite[Theorem 6.1]{DSR} for $f(x,u)=f(u)$) is not necessary to prove the existence of positive solutions to the minimization problem $E_m=\inf _{u \in S_{m}} I(u)$.

We would like to emphasize that a wide class of nonlinearities $f$
satisfy assumptions \eqref{f1}, \eqref{f3} and \eqref{f7}. For instance, typical examples include:
$$f(t)=\abs{t}^{s-2}t\quad  \text{ with } 2<s<2+\frac{4}{N},$$
and
$$f(t)=\abs{t}^{s_1-2}t \pm \abs{t}^{s_2-2}t\quad  \text{ with } 2<s_1<2+\frac{4}{N} \text{ and }s_2>s_1.$$
\begin{theorem}\label{th2}
    Assume that $f$ satisfies \eqref{f1} and \eqref{f7}. $0>E_m>-\infty$ and the mapping $m \mapsto E_m$ is strictly decreasing and continuous. Moreover, if in addition \eqref{f3} hold, then, for all $m>0$, the infimum $E_m$ is achieved by some positive Schwarz symmetric (as defined in Definition \ref{defss}) function $u \in S_m$.
\end{theorem}
The rest of the paper is organized as follows. In Section \ref{secpre}, we introduce some basic facts on lattice graphs and present a discrete Schwarz rearrangement on lattice graphs due to \cite{DSR}.  In Section \ref{secproof}, using the discrete Schwarz rearrangement, we provide the proofs of Theorems \ref{th1} and \ref{th2}.

\section{Preliminaries}\label{secpre}
In this section, we recall some useful preliminaries.
\begin{proposition}\cite[Lemma 2.1]{Huang}
  \label{prointer}
Suppose $u\in \ell^{p}(\mathbb{Z}^{N})$, then $\norm{u}_q \leq \norm{u}_p$ for all $1 \leq p \leq q \leq \infty$.
\end{proposition}
The equivalence of $H^1(\mathbb{Z}^{N})$ and $\ell^2(\mathbb{Z}^{N})$ is well known, for a detailed proof, please refer to \cite{He,HX1}.
\begin{proposition}\label{lemequiv}
The spaces $H^1(\mathbb{Z}^{N})$ and $\ell^2(\mathbb{Z}^{N})$ are equivalent.
\end{proposition}

To prove our main results, we introduce the discrete Schwarz rearrangement on $\mathbb{Z}^{N}$. For more details, please refer to \cite{DSR}. A nonnegative function $u:\mathbb{Z}^{N} \rightarrow \mathbb{R}_{+}$ is said to be admissible if $u\in C_{0}(\mathbb{Z}^{N})$, where $C_0(\mathbb{Z}^{N}) = \{ u \in C(\mathbb{Z}^{N}) : |\{ x : |u(x)| > t \}| < +\infty$, $\ \forall t > 0 \}$ is the space of functions that vanish at infinity. We write $C_{0}^{+}(\mathbb{Z}^{N})$ be the set of all admissible functions.
\begin{definition}\cite[Definition 4.17]{DSR}\label{defss}
	An admissible function $u \in   C_{0}^{+}\left(\mathbb{Z}^{N}\right)$ is called Schwarz symmetric if $u=R_{\mathbb{Z}^{N}}u$ where $R_{\mathbb{Z}^{N}}u$ is the Schwarz rearrangement of $u$.
\end{definition}
Let
$$\mathcal{S}(\mathbb{Z}^{N}):=\left\{u \in C_{0}^{+}\left(\mathbb{Z}^{N}\right): R_{\mathbb{Z}^{N}} u=u\right\}$$
be the set of all Schwarz symmetric functions on $\mathbb{Z}^{N}$, and let
$$\mathcal{S}^{p}(\mathbb{Z}^{N}):=\mathcal{S}(\mathbb{Z}^{N}) \cap \ell^{p}(\mathbb{Z}^{N}).$$
\begin{proposition}\cite[Proposition 4.15]{DSR}\label{prop4.15}
	Suppose that $u_n \in \mathcal{S}(\mathbb{Z}^N)$ and $u_n \to u$ pointwise, then $u \in \mathcal{S}(\mathbb{Z}^N)$.
\end{proposition}

Let $1 \leq p < q < +\infty$, it follows that $\mathcal{S}^p(\mathbb{Z}^N) \subset \mathcal{S}^q(\mathbb{Z}^N)$, since $\ell^p(\mathbb{Z}^N) \subset \ell^q(\mathbb{Z}^N)$. Moreover, we have the following compact embedding result:
\begin{proposition}\cite[Theorem 4.16]{DSR}\label{thm4.16}
	Let $\{ u_n \}$ be a bounded sequence in $\mathcal{S}^p(\mathbb{Z}^N)$, $p \geq 1$. Then, for all $q > p$, there exists a sub-sequence of $\{ u_n \}$ convergent in $\ell^q(\mathbb{Z}^N)$.
\end{proposition}

\begin{proposition}\cite[Proposition 5.1]{DSR}\label{prop5.1}
	Let $f: \mathbb{R}_+ \to \mathbb{R}_+$, $u \in C_0^+(\mathbb{Z}^N)$ be admissible, $u^* = R_{\mathbb{Z}^N} u$. If $f(0) = 0$, then:
	\begin{equation}
		\sum_{x \in \mathbb{Z}^N} f \big( u(x) \big) = \sum_{x \in \mathbb{Z}^N} f \big( u^*(x) \big).
	\end{equation}
\end{proposition}

From \cite[Theorem 5.15]{DSR}, we have
\begin{proposition}\label{thm5.15}
	Let \( u \in \ell^2(\mathbb{Z}^N) \) be non-negative, \( u^* = R_{\mathbb{Z}^N}u \). Then
	\begin{equation}
		\|\nabla u^*\|_2 \leq \|\nabla u\|_2.
	\end{equation}
\end{proposition}

\section{Proof of Theorems \ref{th1} and \ref{th2}}\label{secproof}
We first prove that $I(u)$ is bounded from below on $S_m$ and $E_m \leq 0$ for all $m>0$.
\begin{lemma}\label{lembounded}
Assume that \eqref{f1} holds. Then,  for all $m>0$, $-\infty<E_m\leq 0$.
\end{lemma}
\begin{proof}
We first prove that $E_m > -\infty $ for all $m>0$. Since $\|u\|_{\ell^{2}(\mathbb{Z}^{N})}=m$, by Proposition \ref{prointer}, we have $\|u\|_{\infty}\le \sqrt{m}$. By \eqref{f1}, there exists a constant $C>0$ such that
\begin{equation*}
F(t)\le Ct^{2}, \quad \forall |t|\le \sqrt{m}.
\end{equation*}
Hence, for any $u\in S_{m}$,
$$\int_{\mathbb{Z}^N} F(u)\, d\mu\leq Cm,$$
and thus,
$$
I(u) = \frac{1}{2} \int_{\mathbb{Z}^N}|\nabla u|^2 \,d\mu-\int_{\mathbb{Z}^N} F(u)\,d\mu\geq-\int_{\mathbb{Z}^N} F(u)\,d\mu \geq -Cm,
$$
from where it follows that $E_m \geq -Cm$.

Now, we prove that $E_m \leq 0$ for all $m >0$. Let
\[
    w_n(x):=
     \begin{cases}
       \sqrt{m}n^{-\frac{N}{2}},  & x \in \left[-\frac{n}{2}, \frac{n}{2}\right)^N \cap \mathbb{Z}^N,\\
       0, & else.
      \end{cases}
\]
Then $\int_{\mathbb{Z}^N}\abs{w_n}^{2} = m$ and $\norm{w_n}_\infty = \sqrt{m}n^{-\frac{N}{2}}$. By \eqref{f1}, for any $\varepsilon>0$, there exists $t_\varepsilon>0$ such that
$$
\abs{F(t)} \leq \varepsilon t^2, \quad \forall \abs{t}< t_\varepsilon.
$$
Then, for $n \geq 1$ large enough,
\[
    I(w_n) \leq \frac{m}{2}(2Nn^{N-1}n^{-N}) + m\varepsilon= Nmn^{-1} + m\varepsilon.
\]
Thus $E_m \leq 0$.
\end{proof}
Now, we provide the the proofs of statements (i) and (ii) of Theorem \ref{th1}.
\begin{lemma}\label{lemnonin0}
Assume that \eqref{f1} holds. Then, the following subadditivity inequality holds
   \begin{equation}\label{eqsa}
    E_{a+b}\leq E_a+E_b, \quad \forall a,b>0,
\end{equation}
and the mapping $m \mapsto E_m$ is non-increasing.
\end{lemma}
\begin{proof}
    We first prove the subadditivity inequality \eqref{eqsa}.
Let $\{u^a_n\}$ and $\{u^b_n\}$ be the minimizing sequences of minimization problems $E_a$ and $E_b$, respectively. Then, for any $\varepsilon > 0$, there exists $N_\varepsilon \in \mathbb{N}^+$ such that, for all $n \geq N_\varepsilon$,
$$
I(u_n^a)<E_a+\frac{\varepsilon}{6} \quad \text{and} \quad I(u_n^b)<E_b+\frac{\varepsilon}{6}.
$$
For $R>0$, set
$$
\eta_R(x):=      \begin{cases}
       1,  & x \in \left[-\frac{R}{2}, \frac{R}{2}\right)^N \cap \mathbb{Z}^N,\\
       0, & else.
      \end{cases}
$$
 Fix $\varepsilon > 0$ and $n \geq N_\varepsilon$. Then, by the Lebesgue  Dominated convergence theorem,  we have $\eta_Ru^a_n \to u^a_n$ and $\eta_Ru^b_n \to u^a_n$ as $R \to +\infty$ in $\ell^2(\mathbb{Z}^N)$. Thus, by Proposition \ref{lemequiv}, there exists $R_1 > 0$ such that $\tilde{u}_n^a:= \eta_{R_1}u^a_n \in C_c(\mathbb{Z}^N)$ and $\tilde{u}_n^b:=\eta_{R_1}u_n^b \in C_c(\mathbb{Z}^N)$ satisfying $\norm{\tilde{u}_n^a}^2_2+ \norm{\tilde{u}_n^b}^2_2\leq a+b$ and
\begin{equation}\label{eqnondab}
    I(\tilde{u}_n^a)<E_a+\frac{\varepsilon}{3} \quad \text{and} \quad I(\tilde{u}_n^b)<E_b+\frac{\varepsilon}{3}.
\end{equation}
For $k \in \mathbb{N}^+$, set
\[
    w_k(x):=
     \begin{cases}
       k^{-\frac{N}{2}}\sqrt{a+b-\norm{\tilde{u}_n^a}^2_2- \norm{\tilde{u}_n^b}^2_2},  & x \in \left[-\frac{k}{2}, \frac{k}{2}\right)^N \cap \mathbb{Z}^N,\\
       0, & else.
      \end{cases}
\]
Repeating the argument in Lemma \ref{lembounded}, we know, there exists $k_1\in \mathbb{N}^+$ large enough, such that
\begin{equation}\label{eqnondw}
I(w_{k_1}) < \frac{\varepsilon}{3}.
\end{equation}
Let $e_1:=(1,0,\cdots,0)^T \in \mathbb{Z}^N$. Then,  for $K \in \mathbb{N}^+$ and $K$ sufficiently large, we conclude that $\tilde{u}_{n,K}^a:=\tilde{u}_n^a(x+Ke_1)$, $\tilde{u}_{n,K}^b:=\tilde{u}_n^b(x-Ke_1)$ satisfying $\operatorname{supp}(\tilde{u}_{n,K}^a) \cap \operatorname{supp}(\tilde{u}_{n,K}^b)= \emptyset$, $\operatorname{supp}(\tilde{u}_{n,K}^a) \cap \operatorname{supp}(w_{k_1})= \emptyset$, $\operatorname{supp}(\tilde{u}_{n,K}^b) \cap \operatorname{supp}(w_{k_1})= \emptyset$, and
$$
\int_{\mathbb{Z}^{N}}|\nabla (\tilde{u}_{n,K}^a + \tilde{u}_{n,K}^b +w_{k_1})|^{2}\,d\mu = \int_{\mathbb{Z}^{N}}|\nabla \tilde{u}_{n,K}^a|^{2}\,d\mu+\int_{\mathbb{Z}^{N}}|\nabla \tilde{u}_{n,K}^b |^{2}\,d\mu+\int_{\mathbb{Z}^{N}}|\nabla w_{k_1}|^{2}\,d\mu.
$$
Hence, $\tilde{u}_{n,K}^a + \tilde{u}_{n,K}^b +w_{k_1} \in S_{a+b}$, and thus, by \eqref{eqnondab} and \eqref{eqnondw},
$$
E_{a+b} \leq I(\tilde{u}_{n,K}^a + \tilde{u}_{n,K}^b +w_{k_1}) = I(\tilde{u}_n^a) + I(\tilde{u}_n^b) + I(w_{k_1})< E_a + E_b +\varepsilon.
$$
This proves \eqref{eqsa}.

Now, we prove the mapping $m \mapsto E_m$ is non-increasing. By \eqref{eqsa} and Lemma \ref{lembounded}, for any $m_2>m_1>0$, we have
$$
 E_{m_1}-E_{m_2}\geq- E_{m_2-m_1}\geq 0.
$$
Thus, the mapping $m \mapsto E_m$ is non-increasing.
\end{proof}
\begin{lemma}\label{lemnonin}
Assume that \eqref{f1} holds. Then, the mapping $m \mapsto E_m$ is continuous, and  there exists a uniquely determined number $m^* \in[0, +\infty]$ such that
$$
E_m=0 \text { if } 0<m \leq m^*( 0<m<+\infty\text{ when } m^*=+\infty), \quad \text{ and }\quad E_m<0 \text { if } m>m^*.
$$
\end{lemma}
\begin{proof}
\iffalse We first prove that the mapping $m \mapsto E_m$ is left-continuous.  Fix $m>0$. Let $\{u_n\}$ be the minimizing sequence of problem $E_{m}$. Then, by Lemma \ref{lemnonin0}, for all $\delta\in (0,\frac{m}{2})$,
$$
\begin{aligned}
    0\le E_{m-\delta}-E_{m} & \leq I\left(\frac{\sqrt{m-\delta}}{\sqrt{m}}u_n\right)-I(u_n) + o_n(1)\\
    &=\frac{\frac{m-\delta}{m}-1}{2}\int_{\mathbb{Z}^{N}}|\nabla u_n|^{2}\,d\mu  -\int_{\mathbb{Z}^N} F\left(\frac{\sqrt{m-\delta}}{\sqrt{m}}u_n\right) \, d\mu +\int_{\mathbb{Z}^N} F( u_n) \, d\mu+ o_n(1)\\
    & \leq \int_{\mathbb{Z}^N} F( u_n) \, d\mu-\int_{\mathbb{Z}^N} F\left(\frac{\sqrt{m-\delta}}{\sqrt{m}}u_n\right) \, d\mu + o_n(1)
\end{aligned}
$$
For each $n \in \mathbb{N}^+$, since $\frac{m-\delta}{m}u_n \to u_n$ as $\delta \to 0^+$ in $H^1(\mathbb{Z}^N)$, we have
$$
\int_{\mathbb{Z}^N} F( u_n) \, d\mu-\int_{\mathbb{Z}^N} F\left(\frac{\sqrt{m-\delta}}{\sqrt{m}}u_n\right) \, d\mu \to 0 \quad \text{as } \delta \to 0^+.
$$
Thus, we obtain $E_{m-\delta} \to E_m$ as $\delta \to 0^+$. Since $m>0$ is arbitrary, we conclude that the mapping $m \mapsto E_m$ is left-continuous. \fi
We first prove that the mapping $m \mapsto E_m$ is right-continuous.  Define
$$
D_{m}:=\{u \in H^1(\mathbb{Z}^N): \norm{u}^2_2\leq m\}.
$$
Fix $m>0$.
By \eqref{f1}, there exists $C>0$ such that
\begin{equation*}
\abs{f(t)}\le C\abs{t}, \quad \forall |t|\le \sqrt{2m}.
\end{equation*}
By Proposition \ref{prointer},  for any $u \in D_{2m}$, we have $\norm{u}_\infty \leq \sqrt{2m}$.
For any $u \in S_{m+\delta}$ with $\delta\in (0,\frac{m}{2})$, we have $u \in D_{2m}$, and thus, by the mean value theorem, there exists $\tau \in (\frac{\sqrt{m}}{\sqrt{m+\delta}},1)$ such that
$$
\begin{aligned}
    \left|I(u)-I\left(\frac{\sqrt{m}}{\sqrt{m+\delta}}u\right)\right| &= \left|\left\langle I'(\tau u), \left(1-\frac{\sqrt{m}}{\sqrt{m+\delta}}\right)u\right\rangle\right|\\
    & = \left(1-\frac{\sqrt{m}}{\sqrt{m+\delta}}\right)\left|\tau \int_{\mathbb{Z}^N}\abs{\nabla u}^2\, d\mu - \int_{\mathbb{Z}^N}f(\tau u)u\, d\mu\right|\\
    & \leq \tau\left(1-\frac{\sqrt{m}}{\sqrt{m+\delta}}\right)\left(\int_{\mathbb{Z}^N}\abs{\nabla u}^2\, d\mu +C\int_{\mathbb{Z}^N}\abs{u}^2\, d\mu\right)
\end{aligned}
$$
By Proposition \ref{lemequiv}, we conclude that, there exists $C_1>0$ such that, for all $u \in S_{m+\delta}$ with $\delta\in (0,\frac{m}{2})$,
\begin{equation}\label{eqiumu}
    \left|I(u)-I\left(\frac{\sqrt{m}}{\sqrt{m+\delta}}u\right)\right| \leq C_1\left(1-\frac{\sqrt{m}}{\sqrt{m+\delta}}\right).
\end{equation}
Let $\{u_n^{m+\delta}\}$ be the minimizing sequence of problem $E_{m+\delta}$. By Lemma \ref{lemnonin0}, we have
$$
\begin{aligned}
    E_{m + \delta}\leq E_m &\leq I\left(\frac{\sqrt{m}}{\sqrt{m+\delta}}u_n^{m+\delta}\right)\\
        & \leq I\left(u_n^{m+\delta}\right) + \left|I(u)-I(\frac{\sqrt{m}}{\sqrt{m+\delta}}u)\right|\\
        & \leq I\left(u_n^{m+\delta}\right) + C_1\left(1-\frac{\sqrt{m}}{\sqrt{m+\delta}}\right).
\end{aligned}
$$
Let $n \to +\infty$, we obtain
$$
    E_{m + \delta}\leq E_m \leq E_{m + \delta} + C_1\left(1-\frac{\sqrt{m}}{\sqrt{m+\delta}}\right),
$$
which implies $E_{m+\delta} \to E_m$ as $\delta \to 0^+$. Since $m>0$ is arbitrary, we conclude that the mapping $m \mapsto E_m$ is right-continuous.

Now, we prove that the mapping $m \mapsto E_m$ is left-continuous. Fix $m>0$. Repeating the arguemnt of \eqref{eqiumu}, we conclude that, there exists $C_2>0$ such that, for all $u\in S_m$ and $\delta \in (0,\frac{m}{2})$,
$$
 \left|I\left(\frac{\sqrt{m-\delta}}{\sqrt{m}}u\right)-I(u)\right| \leq C_2\left(1-\frac{\sqrt{m-\delta}}{\sqrt{m}}\right).
$$
Let $\{u_n\}$ be the minimizing sequence of problem $E_{m}$. Then, by Lemma \ref{lemnonin0}, for all $\delta\in (0,\frac{m}{2})$,
$$
\begin{aligned}
    E_{m} \le E_{m-\delta}& \leq I\left(\frac{\sqrt{m-\delta}}{\sqrt{m}}u_n\right)\\
    &\leq  I(u_n) + \left|I\left(\frac{\sqrt{m-\delta}}{\sqrt{m}}u_n\right) -I(u_n)\right|\\
    &\leq I(u_n) + C_2\left(1-\frac{\sqrt{m-\delta}}{\sqrt{m}}\right).
\end{aligned}
$$
Let $n \to +\infty$, we obtain
$$
    E_{m} \le E_{m-\delta} \leq E_{m}  + C_2\left(1-\frac{\sqrt{m-\delta}}{\sqrt{m}}\right),
$$ which implies $E_{m-\delta} \to E_m$ as $\delta \to 0^+$. Since $m>0$ is arbitrary, we conclude that the mapping $m \mapsto E_m$ is left-continuous.  Thus, the mapping $m \mapsto E_m$ is continuous.

Finally, if there exists $m_0$ such that $E_{m_0}<0$, then, since the mapping $m \mapsto E_m$ is non-increasing and continuous, by Lemma \ref{lembounded}, there exists a uniquely determined number $m^* \in[0, m_0)$ such that
$$
E_m=0 \text { if } 0<m \leq m^*, \quad \text{ and }\quad E_m<0 \text { if } m>m^*.
$$
Otherwise, by Lemma \ref{lembounded},  for all $m>0$, $E_m=0$.
This completes the proof of Lemma \ref{lemnonin}.
\iffalse
Finally, we assume that \eqref{f2} holds. Fix $m>0$. For all $\delta > 0$, let $\{u_n^{m+\delta}\}$ be the minimizing sequence of problem $E_{m+\delta}$. Then, we have
$$
\begin{aligned}
   0\geq E_{m+\delta}-E_m &\geq  I(u_n^{m+\delta}) - I\left(\frac{\sqrt{m}}{\sqrt{m+\delta}}u_n^{m+\delta}\right)-\\
   &=\frac{1-\frac{m}{m+\delta}}{2}\int_{\mathbb{Z}^{N}}|\nabla u_n|^{2}\,d\mu  -\int_{\mathbb{Z}^N} F\left(\frac{\sqrt{m-\delta}}{\sqrt{m}}u_n\right) \, d\mu +\int_{\mathbb{Z}^N} F( u_n) \, d\mu+ o_n(1)
\end{aligned}
$$
\fi
\end{proof}

\begin{lemma}\label{lemachieved}
Assume that \eqref{f1}-\eqref{f3} hold. If $E_m<0$, then the infimum $E_m$ is achieved by some positive Schwarz symmetric function $u \in S_m$.
\end{lemma}
\begin{proof}
Let $\{u_{n}\}$ be a minimizing sequence for $E_m$. Then $\{u_{n}\}\subset S_{m}$ and
$$I(u_{n})\to E_m\quad\text{as $n\to \infty$.}$$
It is clear that $|u_{n}|\in S_{m}$. By definition (\ref{tidu}), we have
\begin{equation}\label{abs}
		\int_{\mathbb{Z}^N}|\nabla |u_{n}||^{2}d\mu = \frac{1}{2}\sum_{x \in \mathbb{Z}^N} \sum_{y \sim x}(|u_{n}(y)|-|u_{n}(x)|)^2\le \frac{1}{2} \sum_{x \in \mathbb{Z}^N}\sum_{y \sim x}(u_{n}(y)-u_{n}(x))^2=\int_{\mathbb{Z}^N}|\nabla u_{n}|^{2}d\mu.
	\end{equation}
Thus, by Proposition \ref{prop5.1} and \eqref{f3}, we conclude
$$ E_m\le I(|u_{n}|) \leq I(u_{n})\to E_m.$$
	Therefore, without loss of generality, we can assume that the minimizing sequence $\{u_{n}\}$ for $E_{m}$ always be nonnegative. On the other hand, by the rearrangement inequalities in Proposition \ref{prop5.1} and Proposition \ref{thm5.15}, we have
	\begin{align*}
		\int_{\mathbb{Z}^N}|\nabla u^*|^2d\mu & \leq \int_{\mathbb{Z}^N}|\nabla u|^2d\mu ,
		\\\int_{\mathbb{Z}^N}F(u^{*})d\mu  &=\int_{\mathbb{Z}^N}F(u) d\mu.
	\end{align*}
Hence
	 $u_{n}^{*}\in S_{m}$ and
	$$ E_m\le I(u_{n}^{*}) \leq I(u_{n})\to E_m.$$
	Therefore without loss of generality, we can assume that the minimizing sequence $\{u_{n}\}$ for $E_{m}$ always be Schwarz symmetric.

Let $\{u_n\}$ be a Schwarz symmetric minimizing sequence. Then, by Proposition \ref{prointer}, $\norm{u_n}_\infty\leq \norm{u_n}_2=\sqrt{m}$, and thus, by Bolzano-Weierstrass theorem and diagonal principle, there exists $u \in H^1(\mathbb{Z}^N)$ such that $u_n(x) \to u(x)$ pointwisely in $\mathbb{Z}^N$, up to a subsequence. Thus $F(u_n(x)) \to F(u(x))$ pointwisely in $\mathbb{Z}^N$.
Fix $p>2$. By \eqref{f1}, for any $\varepsilon>0$, we have
\begin{equation}\label{eqfepsilon}
    \abs{f(t)} \leq \frac{\varepsilon}{m} \abs{t} +C_{\varepsilon,p}\abs{t}^{p-1} \quad  \forall t\in [-\sqrt{m},\sqrt{m}],
\end{equation}
where $C_{\varepsilon,p}>0$ is a positive constant.
By Proposition \ref{thm4.16}, we obtain
$$
\lim_{n \to \infty} \int_{\mathbb{Z}^N} \abs{u_n}^p \, d\mu = \int_{\mathbb{Z}^N}  \abs{u}^p\, d\mu.
$$
Then, by Vitali convergence theorem \cite[Corollary 4.5.5]{mt}, we conclude that, for every $\varepsilon > 0$, there exists a finite set $X_\varepsilon \subset \mathbb{Z}^N$ such that
$$
\sup_{n\in \mathbb{N}^+}\int_{\mathbb{Z}^N \backslash X_\varepsilon}\abs{u_n}^p\, d\mu \leq \frac{\varepsilon p}{2C_{\varepsilon,p}} .
$$
By \eqref{eqfepsilon},
$$
\sup_{n\in \mathbb{N}^+}\int_{\mathbb{Z}^N \backslash X_\varepsilon} \abs{F(u_n)}\, d\mu \leq \frac{\varepsilon}{2m}\int_{\mathbb{Z}^N \backslash X_\varepsilon}\abs{u_n}^2\, d\mu + \frac{C_{\varepsilon,p}}{p}\int_{\mathbb{Z}^N \backslash X_\varepsilon}\abs{u_n}^p\, d\mu \leq \frac{\varepsilon}{2} + \frac{\varepsilon}{2}=\varepsilon.
$$
Thus, by Vitali convergence theorem \cite[Corollary 4.5.5]{mt}, we obtain
$$
\lim_{n \to \infty} \int_{\mathbb{Z}^N} F( u_n) \, d\mu= \int_{\mathbb{Z}^N} F(u)\, d\mu.
$$
Since $u_n(x) \to u(x)$ pointwisely in $\mathbb{Z}^N$, we have $\abs{\nabla u_n}(x) \to \abs{\nabla u}(x)$ pointwisely in $\mathbb{Z}^N$. By Fatou's lemma, we obtain
$$
\varliminf_{n\to\infty}\int_{\mathbb{Z}^N}|\nabla u_n|^2 \,d\mu \geq \int_{\mathbb{Z}^N}|\nabla u|^2 \,d\mu.
$$
Hence
$$E_m= \lim_{n \to \infty}I(u_{n})\geq I(u).$$

By the weak lower semi-continuity of the norm $\ell^2$, we have
$$
\int_{ \mathbb{Z}^N} u^2(x)\,d\mu \leq m.
$$
Since $E_{m} < 0$ and $F(0) = 0$, it follows that $u \neq 0$. Set $t = \frac{m}{\|u\|_{2}^{2}}$, then $t \geq 1$. On the other hand, by \eqref{f2} and the strict negativity of $E_{m}$, we observe that
$$
E_{m} \leq I(tu) \leq t^2 I(u) \leq t^2 E_{m}
$$
which implies $t \leq 1$. Therefore, we conclude that $t=1$, i.e. $\|u\|_{2}^{2}=m$, and hence $E_{m}$ is achieved at $u\ge 0$. Consequently, there exists a Lagrange multiplier $\lambda$ such that
	\begin{align*}
		I'(u)=\lambda G'(u ),
        \end{align*}
where $G(u)=\int_{\mathbb{Z}^{N}}u^{2}\,d\mu$. Thus, for any $ \phi \in C_{c}(\mathbb{Z}^{N})$, one has
$$\int_{\mathbb{Z}^N} (\Gamma(u, \phi) - f(u) \phi) \,d\mu = \lambda\int_{\mathbb{Z}^N} u \phi \,d\mu.
$$
Integration by parts yields
$$\int_{\mathbb{Z}^{N}} (-\Delta u - f(u)) \phi \,d\mu = \lambda\int_{\mathbb{Z}^{N}} u\phi \,d\mu.
$$
For any given \( y \in \mathbb{Z}^{N} \), choose the test function $\phi = \delta_y$, where \( \delta_y(x)=\begin{cases}
       1,  &x=y,\\
       0, &x \neq y,
      \end{cases} \)

we obtain
$$
-\Delta u(y) = \lambda u(y) + f(u(y)).
$$
Since $ y \in \mathbb{Z}^{N}$ is arbitrarily, we have $u$ is a pointwise nonnegative solution to equation \eqref{eqpm}. Moreover, by Proposition \ref{prop4.15}, we conclude that $u$ is Schwarz symmetric.

Finally, we prove that the $u>0$. If $u(x_0)=0$ for some fixed $x_0 \in \mathbb{Z}^N$, then equation \eqref{eqpm} yields that $u(x) = 0$ for all $x \sim x_0$. Since $\mathbb{Z}^N$ is connected, we obtain $u(x) \equiv 0$, which leads to a contradiction.
\end{proof}
Finally, we proceed to prove the statements (iii)-(v) of Theorem \ref{th1}.
\begin{lemma}\label{lemem}
  Assume that \eqref{f1} holds. Let $m^*$ be defined as in Lemma \ref{lemnonin}. Then the following statements hold.
  \begin{enumerate}[label=(\roman*)]
  \item If \eqref{f4} holds, then $E_{\zeta^2}<0$ and $m^*\in [0,\zeta^2)$.
  \item If \eqref{f5} holds, then $m^*\in (0,+\infty]$.
  \item If \eqref{f6} holds, then for all $0<m<m^*$, $E_m$ is not achieved.
  \end{enumerate}
\end{lemma}
\begin{proof}
   (i) Assume that \eqref{f1} and \eqref{f4} hold. Set $$
    \delta_0(x):=
     \begin{cases}
       1,  &x=0,\\
       0, &x \neq 0.
      \end{cases}
    $$
Then $\zeta \delta_0 \in S_{\zeta^2}$. By \eqref{f4}, we have
$$
E_{\zeta^2}\leq I(\zeta \delta_0)=2N{\zeta^2}-F(\zeta)= 2N\zeta^2(1-\frac{F(\zeta)}{2N\zeta^2})<0.
$$
By Lemma \ref{lemnonin0} and \ref{lemnonin}, we know that the mapping $m \mapsto E_m$ is non-increasing and $E_{m^*}=0$. Thus, $m^*\in [0,\zeta^2)$.

(ii) Next, we use the following discrete Gagliardo-Nirenberg inequality
\begin{equation}\label{eqgn}
    \norm{u}^{2+\frac{4}{N}}_{2+\frac{4}{N}} \leq C_N\norm{\nabla u}_2^2\norm{u}^{\frac{4}{N}}_{2},\quad \forall u \in H^1(\mathbb{Z}^N),
\end{equation}
where $C_N>0$ depends only on $N$. For $N=2$, we refer to inequality (4.11) with $\sigma=1$ in \cite{Weinstein}; For $N \geq 3$, we recall the following famous Sobolev inequality on lattice graphs due to \cite[Lemma 2]{Stefanov} and \cite[Theorem 3.6]{hua2015time},
$$
\int_{\mathbb{Z}^N}|u|^{2^*}\,d\mu \leq S_N\left(\int_{\mathbb{Z}^N}|\nabla u|^{2}\, d\mu \right)^\frac{2^*}{2},\quad \forall u \in \ell^2(\mathbb{Z}^N),
$$
where $2^*=\frac{2N}{N-2}$ and $S_N>0$  depends only on $N$. Hence, by Proposition \ref{lemequiv} and H\"{o}lder's inequality, we obtain
$$
\begin{aligned}
    \int_{\mathbb{Z}^N}|u|^{2+\frac{4}{N}}\,d\mu =\int_{\mathbb{Z}^N}|u|^{2}|u|^{\frac{4}{N}}\,d\mu &\leq \left(\int_{\mathbb{Z}^N}|u|^{2\cdot\frac{2^*}{2}}\,d\mu\right)^\frac{2}{2^*}\left(\int_{\mathbb{Z}^N}|u|^{\frac{4}{N}\cdot\frac{N}{2}}\,d\mu\right)^\frac{2}{N}\\ &\leq S_N^\frac{2}{2^*} \norm{\nabla u}^{2}_2 \norm{u}^{\frac{4}{N}}_{2}.
\end{aligned}
$$
Now assume that \eqref{f1} and \eqref{f5} hold. By \eqref{f5}, we know that there exist $C>0$ and $t_1 >0$ such that
$$
F(t) \leq C\abs{t}^{2+\frac{4}{N}}, \quad \forall \abs{t} < t_1.
$$
Let $0<m<t_1^2$. Then, by Proposition \ref{prointer}, we know that,  $\norm{u}_\infty <t_1$ for all $u \in S_m$. Thus, by \eqref{eqgn}, it follows that for all $u \in S_m$,
$$
I(u) = \frac{1}{2}\norm{\nabla u}_2^2 - \int_{\mathbb{Z}^N}F(u)\,d\mu\geq \frac{1}{2}\norm{\nabla u}_2^2 -C\norm{u}^{2+\frac{4}{N}}_{2+\frac{4}{N}} \geq \frac{1}{2}\norm{\nabla u}_2^2 -CC_Nm^{\frac{2}{N}}\norm{\nabla u}_2^2.
$$
Hence, for $m>0$ small enough, we have $E_m\geq 0$. By Lemma \ref{lemnonin}, we conclude $m^*\in (0,+\infty]$.

(iii) Finally, assume that \eqref{f1} and \eqref{f6} hold. Argue by contradiction, assume that $E_{m}$ is achieved for some $m \in (0,m^{*})$. Then, there exists $u\in S_{m}$ such that
$
I(u)=E_{m}=0
$.
Clearly, $u \neq 0$ and $v=\frac{\sqrt{m^{*}}}{\sqrt{m}}u\in S_{m^{*}}$ and then, by \eqref{f6},
$$
E_{m^{*}}\le I(v)= \frac{m^{*}}{m}\int_{\mathbb{Z}^{N}}|\nabla u|^{2}\,d\mu-\int_{\mathbb{Z}^{N}}F(\frac{\sqrt{m^{*}}}{\sqrt{m}}u)\,d\mu< \frac{m^{*}}{m}I(u)=0.
$$
This contradicts the fact that  $E_{m^*} = 0$. This completes the proof of Lemma \ref{lemem}.
\end{proof}

It is the position to provide the proofs of Theorems \ref{th1} and \ref{th2}.
\begin{proof}[Proofs of Theorems \ref{th1} and \ref{th2}]
From Lemma \ref{lembounded}-\ref{lemem}, we complete the proof of Theorem \ref{th1}.

Now, we prove Theorem \ref{th2}. Assume that \eqref{f1} and \eqref{f7} hold. Repeating the argument of \cite[Lemma 3]{Stefanov} and \cite[Remark 6.2]{DSR}, for all $k \in \mathbb{N}^+$, we obtain $u_k$ satisfies $u_k \in S_m$,
\begin{equation}\label{eqgnc}
     \lim_{k \to +\infty}\norm{u_k}_\infty=0, \quad \norm{\nabla u}_2^2 \leq C_1k^{-2}\quad \text{and }\quad\norm{u}_{2+\frac{4}{N}}^{2+\frac{4}{N}} \geq C_2k^{-2}
\end{equation} for some constants $C_1,C_2 >0$ depending on $m$ and $N$.
On the other hand, by \eqref{f7}, for any $G>0$, there exists $t_G>0$ such that
$$
\abs{F(t)} \geq G t^{2+\frac{4}{N}}, \quad \forall \abs{t}< t_G.
$$
For $k \in \mathbb{N}^+$ sufficiently large, we have $\norm{u_k}_\infty <t_G$, and thus, by \eqref{eqgnc}, we conclude
$$
I(u_k) = \frac{1}{2}\norm{\nabla u_k}_2^2 - \int_{\mathbb{Z}^N}F(u_k)\,d\mu\leq \frac{1}{2}\norm{\nabla u_k}_2^2 -G\norm{u_k}^{2+\frac{4}{N}}_{2+\frac{4}{N}} \leq \frac{C_1k^{-2}}{2} -GC_2k^{-2},
$$
which implies that $E_m<0$ for all $m>0$. Then, the subadditivity inequality \eqref{eqsa} implies the mapping $m \mapsto E_m$ is strictly decreasing. Moreover, by Lemma \ref{lemnonin}, the mapping $m \mapsto E_m$ is continuous.

 Next, assume in addition that \eqref{f3} holds. By repeating the argument in Lemma \ref{lemachieved}, we obtain a Schwarz symmetric function $u\in H^1(\mathbb{Z}^N)\backslash\{0\}$ satisfying $\norm{u}^2_2\leq m$ and $I(u) \leq E_m$. Argue by contradiction, suppose that $\norm{u}^2_2=m_1<m$. Then,
$$
E_{m_1}\leq I(u) \leq E_m,
$$
which leads to a contradiction, since the mapping $m \mapsto E_m$ is strictly decreasing. Hence, $\norm{u}^2_2=m$, and the infimum is attained. Repeating the argument in Lemma \ref{lemachieved}, we conclude that $u$ is positive. This completes the proof of  Theorem \ref{th2}.
\end{proof}
\subsection*{Conflict of interest}

On behalf of all authors, the corresponding author states that there is no conflict of interest.

\subsection*{Data Availability Statements}
Data sharing not applicable to this article as no datasets were generated or analysed during the current study.

\subsection*{Acknowledgements}
C. Ji is supported by National Natural Science Foundation of China (No. 12171152).
  
\end{document}